\documentclass[12pt]{amsart}

\usepackage[all]{xy}
\usepackage{amsfonts}
\usepackage{amssymb}

\usepackage{amsmath,amssymb,amsthm, latexsym, amscd}

\usepackage{color}

\newtheorem{theorem}{Theorem}[section]
\newtheorem{lemme}[theorem]{Lemma}
\newtheorem{prop}[theorem]{Proposition}

\newtheorem{cor}[theorem]{Corollary}

\theoremstyle{definition}
\newtheorem{definition}[theorem]{Definition}
\newtheorem{remark}[theorem]{Remark}
\newtheorem{ex}[theorem]{Example}

\def\<{\langle}
\def\>{\rangle}

\def\a{\alpha}
\def\b{\beta}

\def\te{\theta}

\def\z{{\zeta}}

\def\te{\theta}
\def\Te{\Theta}
\def\C{{\mathbb C}}
\def\N{{\mathbb N}}
\def\R{{\mathbb R}}
\def\Z{{\mathbb Z}}
\def\Q{{\mathbb Q}}

\def\H{{\mathcal H}}
\def\B{{\mathcal B}}

\def\M{{\mathcal M}}
\def\O{{\mathcal O}}

\def\dist{\mathop{\rm dist}\nolimits}

\def\sys{\mathop{\rm sys}\nolimits}

\def\ln{\operatorname{ln}}

\def\vol{\operatorname{vol}}
\def\Vol{\operatorname{Vol}}
\def\V{\operatorname{Vol}}
\def\mass{\operatorname{mass}}

\def\T{{\mathbb T}}
\def\1{\mathbf 1}

\def\N{{\mathbb N}}

\subjclass{Primary 53C20, 53C99, 30B50; Secondary 05C35}


\thanks{}

\begin{document}

\title{On systolic zeta functions}

\author{Ivan Babenko and Daniel Massart}

\thanks{Partially supported by the grants RFSF 10-01-00257-a, and ANR Finsler.}

\address{Institut Montpelli\'erain Alexander Grothendieck,  CNRS,  Univ. Montpellier, France}
 \email{ivan.babenko@umontpellier.fr,
daniel.massart@umontpellier.fr}

\maketitle

\begin{abstract}
We define Dirichlet type  series associated with homology length
spectra of Riemannian, or Finsler, manifolds, or polyhedra, and
investigate some of their analytical properties. As a consequence
we obtain an inequality analogous to Gromov's classical
intersystolic inequality, but taking the whole homology length
spectrum into account rather than just the systole.
\end{abstract}

\section{Introduction}
Let $(M, g)$ be a closed  $m$-dimensional Riemannian manifold. Let $l$ be
the shortest possible length of a closedl geodesic whose homology
class is non-trivial in $H_1(M,\R)$. This length is called
homological systole of $(M, g)$ and denoted by  $l =
\text{sys}_H(M,g)$. It carries important information about the
manifold $(M,g)$. As was discovered by Gromov (\cite{Gromov83}),
under topological conditions on M to be explained below, the
following inequality holds for all Riemannian metrics g on M :

\begin{equation}\label{eq:equation.de.base}
\sigma_m\big{(}\text{sys}_H(M, g)\big{)}^m \leq \vol(M, g),
\end{equation}
where $\sigma_m$ is a universal constant only depending on the
dimension $m$, whose optimal value is unknown. Now let us explain
the topological condition under which Inequality  (1) holds.
Recall that for any manifold M there is a map, called {\it
characteristic map}, unique modulo homotopy,
\begin{equation}\label{eq:application.caracteristique.bis}
f: M \longrightarrow \T^{b_1}
\end{equation}
where $b_1=b_1(M)$ is the first Betti number of M and $T^{b_1}$ is
the $b_1$-dimensional torus, such that the induced map
$$
f_*: H_1(M, \Z)/\text{Tors} \longrightarrow H_1(T^{b_1}, \Z),
$$
where Tors means the torsion subgroup of $H_1(M,\Z)$, is an
isomorphism of $\Z$-modules.

 Let  $\a[M]$ denote $ f_*([M])\in
H_m(M, \Z_{\star})$,where  $[M]$ is the fundamental class of  $M$
and $\Z_{\star} = \Z$ when $M$ is orientable and $\Z_2$ otherwise.
A sufficient condition for (\ref{eq:equation.de.base}) to hold
for any metric  $g$ is that $\a[M]\neq 0$. It can be shown to be
also necessary.

In this paper, instead of just the homological systole, we shall
consider what we call the {\it homology length spectrum} of $(M,
g)$. For any $\te \in H_1(M,\Z)/\text{Tors}$, denote by $l_{\te}$
the smallest possible length of a closed geodesic in the homology
class $\te$. Denote by $\Te= H_1(M, \Z)/\text{Tors} \setminus \{0\}$
the set of non-trivial homology classes and define the homology
length spectrum of $(M, g)$ as $L_{\Te} = L_{\Te}(g) =
\{l_{\te}\}_{\te \in \Te}$. Of course,  $\text{sys}_H(M, g) = \inf
L_{\Te}$. The goal of this paper is to extract geometric and
topological information from the homology length spectrum as it
was done from the systole. The information carried by the
homological length spectrum $L_{\Te}$ is conveniently encoded in
the formal series

\begin{equation}\label{eq:zeta.bis}
\z_{sys}(z) = \mathop{\sum}\limits_{\te \in
\Te}\frac{1}{l_{\te}^z}
\end{equation}
which we call {\it systolic zeta function} of $(M, g)$. In Chapter
\ref{ch:zeta} we show that the formal series (\ref{eq:zeta.bis})
converges in a natural sense and enjoys interesting analytical
properties.

To explain our main results we need to define the stable norm of
$(M, g)$. Any metric $g$ on $M$ induces a norm $\|\cdot\|_{st}^g$
on $H_1(M, \R)$ called stable norm (see \cite{Federer74} and
Chapter \ref{ch:zeta} for more detail). Denote by $B_g(1)$  the
unit ball of the stable norm. Let $V (g)$  be the volume of
$B_g(1)$ with respect to the Haar measure on $H_1(M, \R)$,
normalised in such a way that the quotient of $H_1(M, \R)$ by the
lattice $H_1(M, \Z)/\text{Tors}$  has volume 1.

\begin{theorem}\label{th:A}
For any Riemannian manifold $(M, g)$, the series
(\ref{eq:zeta.bis})  converges for all  $z$ such that
$\text{Re}(z) > b_1(M)$, and diverges when $Re(z) < b_1(M)$. The
function $\z_{sys}(z)$ is holomorphic in the half-plane
$\text{Re}(z) > b_1(M)$.

\noindent Furthermore the  function $\z_{sys}(z)$ extends
analytically to the half-plane  $\text{Re}(z) > b_1-1$, with a
simple pole at $z=b_1$ with residue $Res_{b_1}(\z_{sys}(z))
=b_1(M)V(g)$.
\end{theorem}
This result comes as a particular case of Theorem
\ref{th:zeta.syst.A}, proven in Chapter \ref{ch:proprietes.zeta}, which holds
for any Riemannian or Finsler  polyhedron.

The analytical properties of $\z_{sys}(z)$ are reflected in the
homology length spectrum, which enables us to prove an inequality
analogous to (\ref{eq:equation.de.base}). Furthermore our approach
allows us to estimate the constants that appear in the inequality.

\begin{cor}\label{th:B}
Let $M$ be  an orientable $m$-dimensional differentiable manifold
with first Betti number $b_1$. For any Riemannian metric $g$ on
$M$ we have the following inequality :

$$
\bigg{(}\frac{b_1v_{b_1}V[M]}{Res_{b_1}(\z_{sys}(z))}\bigg{)}^m
\leq \big{(}\vol(M, g)\big{)}^{b_1}.
$$
\end{cor}
Here $V [M]$  is a topological invariant called algebraic volume,
which is defined in Chapter \ref{ch:isoperimetrique}. In the
particular case $m = b_1$ it is just  the degree of the
characteristic map: $V[M]=|\text{deg}(f)|$. The number $v_{n}$ is
a universal constant which comes from the solution of an
optimization problem in $n$-dimensional convex geometry. Although
the exact value of $v_{n}$ is unknown when $n \geq 3$, there are
explicit estimates which are asymptotically correct when $n$ goes
to infinity. See Chapter  \ref{ch:isoperimetrique} for more detail
and the proof of Corollary \ref{th:B}.

The classical Riemann zeta function $\z(z) =
\mathop{\sum}\limits_{n=1}^{\infty}{1\over n^z}$ appears naturally
in the previous context  as $1/2$ times the systolic zeta  function of
the simplest manifold, the circle $S^1$. In this case for a metric
$g$ of unit volume (or  length) we have
$$
\z_{sys}(z) = 2\z(z).
$$
Another well known number theoretic zeta function is the so called
Hurwitz zeta function
\begin{equation}\label{eq:hurwitz.zeta}
\z(z; q) = \mathop{\sum}\limits_{n=0}^{\infty}{1\over (q+n)^z},
\end{equation}
where $q$  is a complex parameter with  Re$(q) > 0$. See \cite{Apostol76}
for details and analytic properties as well as connections of
Hurwitz zeta functions for rational parameter $q$ with  Dirichlet
$L$-functions.

 Hurwitz zeta functions  appear naturally in the systolic
setting as well. More precisely in Chapter \ref{ch:zeta.stable} we
define the stable systolic zeta functions for Riemannian
polyhedra. Weighted graphs are the simplest examples. A graph is
called  combinatorial if the weight of each edge is equal to 1.

In Chapter \ref{ch:zeta.stable},  the  stable systolic zeta function
for a combinatorial graph is explicitly calculated as a linear
combination of Hurwitz zeta functions, see Theorem
\ref{th:zeta.graphe} for a precise statement. Such a
representation of the stable systolic zeta function for a graph
provides an analytic extension of the stable systolic zeta
function to the whole  complex plane as a meromorphic function,
see Chapter \ref{ch:zeta.stable} for details.

All the  results in this paper are presented in the more familiar
Riemannien context. Nevertheless,  up to Chapter 6,  all of them can
be directly translated to the Finsler setting. Moreover a stable
systolic zeta function  can be seen as a Finsler zeta function  on the corresponding
flat finsler torus. To be expressed in the Finsler language, the
results of Chapter 6 would require more precise statements and a
suitable choice of   Finsler volume.


\section{Systolic zeta function}\label{ch:zeta}

Let us consider a countable set $\Te$ and a map $l: \Te
\longrightarrow \R_+$. Let
$$
L_{\Te} = \{l_\te\}_{\te \in \Te}
$$
(or $L$ for short) be the  graph of this map, seen as a family of
positive numbers, not necessarily pairwise disjoint, indexed by
$\Te$. We say that  $L_{\Te}$ is a  $\Te$-{\it marked spectrum}.
If we order  $L_{\Te}$ in a natural way, taking multiplicities
into account, we get the ordered spectrum
$$
|L_{\Te}| = \{(l_i, a_i)\}_{i=1}^{\infty} , \ \ l_1 < l_2 < \dots,
$$
where the $a_i \in \N$ are the respective multiplicities of the
elements $l_i, i \in \N$. Let us call  {\it  zeta function } of
the family $L$ the formal series
\begin{equation}\label{eq:zeta}
\z_L(z) = \mathop{\sum}\limits_{\te \in \Te}\frac{1}{l_{\te}^z}.
\end{equation}

Denote
$$
\Te_{\leq t} = \{\te \in \Te \big{|} l_{\te} \leq t\} \ \
\text{and} \ \ \Te_{>t} = \Te \setminus \Te_{\leq t}.
$$
Let us say that the series  (\ref{eq:zeta}) converges for some  $z
\in \C$ if the sum
$$
\mathop{\sum}\limits_{\te \in \Te_{\leq t}}\frac{1}{l_{\te}^z}
$$
 is finite for all $t > 0$ and has a limit when
$t\longrightarrow \infty$. Observe that  this notion of cenvergence
coincides with the {\it greedy convergence} introduced in
\cite{Shchepin11}, and that the convergence of the series
(\ref{eq:zeta}) is equivalent to the convergence of the classical
Dirichlet series
$$
\z_{|L|}(z) =
\mathop{\sum}\limits_{i=1}^{\infty}\frac{a_i}{l_i^z},
$$
where $(l_i, a_i)$ are the elements of the marked spectrum
$|L_{\Te}|$. This explains several analytical properties of the
function  (\ref{eq:zeta}). It is easily seen that if the
cardinality of  $\Te_{\leq t}$ satisfies
\begin{equation}\label{eq:2.convergence}
|\Te_{\leq t}| = \O(t^b)
\end{equation}
then the series (\ref{eq:zeta}) converges in the half-plane $Re(z)
> b$ and the function  $\z_L(z)$ is holomorphic therein.

\begin{ex}\label{ex:zeta.systolique}
Let us consider a finite simplicial polyhedron $P$ and set
\begin{equation}\label{eq:3.Omega.syst}
\Te = \big{\{}(H_1(P, \Z)/\text{Tors}) \setminus \{0\}\big{\}}.
\end{equation}
Let  $g$ be a polyhedral Riemannian metric on $P$ (see e.g.
\cite{Babenko02}),  and for any  $\te \in \Te$ let us define
$l_{\te}$ as the shortest possible length of a closed geodesic in
the homology class $\te$. The family $L(P,g) = L_{\Te}$ is
well-defined and represents the homology length spectrum of $(P,
g)$. Denote by $\z_{\sys(P,g)}(z)$ (or $\z_{\sys}(z)$ for the sake of
brevity) the corresponding series   (\ref{eq:zeta}) for  $L_{\Te}$.

\end{ex}
\begin{definition}
The function  $\z_{\sys}(z)$ defined by  (\ref{eq:zeta}) for the
family $L(P,g)$ is called  {\it systolic zeta function} of the
Riemannian polyhedron $(P, g)$.
\end{definition}


\section{Stable zeta function}\label{ch:zeta.stable}


Let  $\B = (\R^b, \|\cdot \|)$ be a Banach space, and let $\Gamma
\subset \R^b$ be a lattice. Let us normalise the Haar measure on
$\R^b$  so that $\Gamma$ has volume 1. Set $\Te = \Gamma \setminus
\{0\}$ et $L_{\Te} = \{l_{\te} = \|\te\|, \te \in \Te\}$, then
$\Te_{\leq t} = B(t) \cap \Te$ where $B(t)$ is the ball of radius
$t$  centered at $0$ in $\B$. Denote $\z_{\B}(z)$ the corresponding
zeta function.

For instance,  if  $\B_{\nu} = (\R^2, \|\cdot \|_{\nu})$ with  $\nu
= \{1 , \infty\}$, then $\z_{\B_{\infty}}(z) = 2\z_{\B_1}(z) =
8\z(z-1)$ where  $\z$ is Riemann's zeta function. In higher dimensions
$\z_{\B_{\nu}}(z)$, for $\nu = \{1 , \infty\}$, may also be
expressed in terms of  $\z$ but we shall not dwell on this topic.
In general  $\z_{\B}(z)$ enjoys the following properties.

\begin{prop}\label{prop:zeta_B}
Let  $\B = (\R^b, \|\cdot \|)$ be a Banach space, $\Gamma \subset
\R^b$ a lattice, and  $V = \V (B(1))$  the  volume of the unit
ball of  $\B$. Then the function  $\z_{\B}(z)$ satisfies the following: 

\noindent 1) the series (\ref{eq:zeta})  converges for all $z$
such that $\text{Re}(z) > b$ and diverges if $Re(z) < b$.

\noindent 2) the function $\z_{\B}(z)$ extends holomorphically to
the half-plane $\text{Re}(z) > b-1$, with a simple pole at $z=b$,
with residue $Res_b(\z_{\B}(z)) = bV$.
\end{prop}

\begin{proof}
Consider the following integer sequence
$$
\Big{\{}a(n) = \Big{|}(B(n)\setminus B(n-1)) \cap
\Gamma\Big{|}\Big{\}}_{n=1}^{\infty}
$$
and set
$F(z)=\mathop{\sum}\limits_{n=1}^{\infty}\frac{a(n)}{n^z}$. It is
easily seen that the series  $\z_{\B}(z)$ et $F(z)$ convverge in
the same half-plane. Since
\begin{equation}\label{eq:3.A(t)}
A(t) = \mathop{\sum}\limits_{n \leq t}a(n) = \Big{|}B(t)\cap
\Gamma\Big{|} = Vt^b + \O(t^{b-1}),
\end{equation}
the common half-plane of convergence is $\{ z| Re(z) > b\}$ and $z
= b$ is singular for both functions. This entails 1).

Moreover we have
$$
\z_{\B}(z) - F(z) = \mathop{\sum}\limits_{n=1}^{\infty}
\bigg{(}\mathop{\sum}\limits_{\te \in (B(n)\setminus B(n-1)) \cap
\Gamma}\Big{(}\frac{1}{\|\te\|^z} - \frac{1}{n^z}\Big{)}\bigg{)}.
$$
Applying, for $z \in \R$, the Mean Value Theorem to the function
$$
\begin{array}{rcl}
f_z : \  \R_{+}^{*} & \longrightarrow & \C \\
t & \longmapsto & t^{-z}
\end{array}
$$
we get that, for all $t$, there exists $h(t) \in \left]t, E(t)+1
\right[$, where $E(.)$ is the floor function, such that
$$
f_z(E(t))-f_z(t) = f'_z(h(t)) (1 - \{t\})
$$
where $\{x\}$ is the fractional part of $x$.

Then we have, for  $z \in \R$,
\begin{equation}\label{eq:4}
\z_{\B}(z) - F(z) = z\mathop{\sum}\limits_{n=1}^{\infty}
\bigg{(}\mathop{\sum}\limits_{n-1<\| \te
\|<n,}\frac{1-\{\|\te\|\}}{h(\te)^{z+1}}\bigg{)},
\end{equation}
with  $\|\te\| < h(\te) <[\|\te\|]+1$. Therefore the series
(\ref{eq:4}) converges if  $z > b-1$, thus it is holomorphic in
the half-plane $\text{Re}(z) > b-1$.

Taking  (\ref{eq:3.A(t)}) into account, it remains to apply
Tauber's Theorem (see e.g. \cite{Tenenbaum95}, th\'eor\`eme 2,
{\S}7.1)  to the function $F(z)$. It follows that $F(z)$  extends
to $\text{Re}(z) > b-1$ and
$$
F(z) = \frac{bV + o(1)}{z - b}, \ \ \text{si} \ \ z
\longrightarrow b.
$$
This ends the proof.
\end{proof}
If  $(P, g)$ is a Riemannian polyhedron, let us set $\Gamma =
H_1(P, \Z)/\text{Tors}$. Then  $\Gamma$ embeds canonically  into
$H_1(P, \R)$.We set $\Te(P)=\Gamma \setminus \{0\}$. It is
well known that for all $\te \in \Gamma$ the following limit
\begin{equation}\label{eq:lim.stable}
\mathop{\lim}\limits_{n\rightarrow\infty}\frac{l_{n\te}}{n}
\end{equation}
exists  and defines a norm on $\Gamma$ which extends by
homogeneity and continuity to  $H_1(P, \R)$. This norm is called
{\it mass} or {\it stable norm}, we denote it $\|\cdot \|^g_{st}$
(or $\|\cdot \|_{st}$ for the sake of brevity). See
\cite{Federer74} for several equivalent definitions.

Applying the contruction at the beginning of this section to the
Banach space $\B_{st} = (H_1(P, \R), \|\cdot \|^g_{st})$ we get
the function
\begin{equation}\label{eq:zeta.stable}
\z_{st}(z) = \mathop{\sum}\limits_{\te \in
\Te(P)}\frac{1}{\|\te\|_{st}^z}
\end{equation}
which we call  {\it stable zeta function } of the Riemannian polyhedron
$(P, g)$. Thus it verifies Proposition \ref{prop:zeta_B}.

Weighted graphs are interesting examples of Riemmannian polyhedra
of  dimension 1. Such a graph  is given by a set of vertices $V$,
a set of edges $E$ (we allow loops and multiple edges between two
given vertices, and edges are unoriented), and a weight function
 $w : E \longrightarrow \R_+$. We set $P=(V,E)$, so the weighted graph is $(P,w)$.
Each edge is then endowed with a Riemannian metric such that the
length of an edge $e$ is $w(e)$. Then the weighted graph  $(P,w)$
becomes a  1-dimensional Riemannian polyhedron. We say that  $(P,
w)$ is a {\it combinatorial graph} if $w(e)=1$ for all $e \in E$.
The  only topological invariant of $P$ is its first Betti number
$b = b_1(P) = |E| - |V| + 1$.

The stable norm on $H_1(P, \R)$ only depends on the weight
function $w$ on $P$ and not on the particular choice of the
Riemannian metric  on each edge. Its unit ball $B_{st}^w $ is
always a  $b$-dimensional polytope whose vertices are in
one-to-one correspondance with the simple cycles in $P$, and their
number is bounded above by  $2(2^b - 1)$, see \cite{Babenko06} for
more detail.

In particular, for any weighted graph  $(P, w)$ the stable zeta
function $\z_{st}^w(z)$ satisfies Proposition \ref{prop:zeta_B}.
For a general weight function $w$ we cannot expect additional
analytical properties for  $\z_{st}^w(z)$. However our next
theorem shows that  $\z_{st}^w(z)$ does have interesting
analytical properties when $(P,w)$ is a combinatorial graph. In
the statement, $\zeta(s)$ means, as usual, Riemann's zeta function.

\begin{theorem}\label{th:zeta.graphe}
Let  $P$ be a combinatorial graph with first Betti number  $b$,
and let $V = \V (B_{st})$ be the  volume of the unit ball of the
stable norm in  $H_1(P, \R)$. Then the stable zeta function
$\z_{st}(z)$ of the graph $P$ satisfies: 

\noindent 1) the series (\ref{eq:zeta.stable})  converges  for all
$z$ such that  $\text{Re}(z) > b$ and diverges if  $Re(z) < b$.

\noindent 2) the function  $\z_{st}(z)$ extends analytically to
the whole complex plane $\C$ as a meromorphic function with simple
poles at
 $z=1, 2, \dots , b$, whose residue at
 $z=b$ is $Res_b(\z_{st}(z)) = bV$.

\noindent 3) there exists  $m=m(P) \in \N$ such that
\begin{equation}\label{eq:zeta.graphe.1}
\z_{st}(z) = bV\zeta(z - b +1) +
\mathop{\sum}\limits_{l=0}^{b-1}m^{l-z}\bigg{(}\mathop{\sum}\limits_{k=1}^m
p_{lk}\zeta\Big{(}z-l; {k\over m}\Big{)}\bigg{)},
\end{equation}
where  $\{p_{lk}\} \in \Q$ and $ \zeta(s; q)$ is the Hurwitz zeta function 
defined in (\ref{eq:hurwitz.zeta}).

\end{theorem}
\begin{proof}
Part 1) follows directly from Proposition \ref{prop:zeta_B}.
We prove 3) before 2). If $P$ is a  combinatorial graph,  for all
$\theta \in \Theta(P)$, $\|\theta\|_{st} \in \N$  (see
\cite{Babenko06}), so we have
\begin{equation}\label{eq:zeta.graphe.2}
\z_{st}(z) = \mathop{\sum}\limits_{\theta \in \Theta}{1\over
\|\theta\|_{st}^z} = \mathop{\sum}\limits_{n =
1}^{\infty}{A_n\over n^z} ,
\end{equation}
where  $A_n$ is the number of  points in  $\Theta$ with  norm $n$.
If $B_{st}(t)$ is the ball of radius $t$,  since   the norm of any
$\theta \in \Theta(P)$ is an integer,  we get
\begin{equation}\label{eq:zeta.graphe.A_n}
A_n = |B_{st}(n) \cap \Theta(P)| - |B_{st}(n-1) \cap \Theta(P)|.
\end{equation}
For any combinatorial graph, the unit ball $B_{st}$ of the stable
norm is a $b$-dimensional polytope whose vertices are rational
with respect to the lattice $H_1(P, \Z)$ (see\cite{Babenko06}).
Applying Ehrhart's theorem (\cite{Ehrhart62}), taking the equality
$B_{st}(t) = tB_{st}$ into account, we get a natural number $m$,
which depends only on  $B_{st}$, such that
$$
|B_{st}(n) \cap \Theta(P)| = Vn^b + q(n),
$$
where  $q(n)$ is an  $m$-quasipolynomial  of degree at most $b-1$
with rational coefficients, that is,  $q(n) =
\mathop{\sum}\limits_{l=0}^{b-1}q_l(n)n^l$, where $q_l(n)$ are
$m$-periodic functions. From this and  (\ref{eq:zeta.graphe.A_n})
we deduce
$$
A_n = bVn^{b-1} + p(n),
$$
where  $p(n) = \mathop{\sum}\limits_{l=0}^{b-1}p_l(n)n^l$  is a
$m$-quasipolynomial  of degree at most $b-1$ with rational
coefficients. Observe that the leading coefficient  $p_{b-1}(n)$
averages to zero : $\mathop{\sum}\limits_{k=1}^{m}p_{b-1}(k)=0$.
This is because $p(n)=q(n)-q(n-1)$ so, setting $q(n) =
\mathop{\sum}\limits_{l=0}^{b-1}q_l(n)n^l$, we get
$$
p(n)= q(n)-q(n-1)=  \mathop{\sum}\limits_{l=0}^{b-1}q_l(n)n^l -
\mathop{\sum}\limits_{l=0}^{b-1}q_l(n-1) (n-1)^l
$$
whose leading coefficient is $q_{b-1}(n)-q_{b-1}(n-1)$, which
averages to zero. We shall use this fact in the proof of part 2).

Now  we get from (\ref{eq:zeta.graphe.2})
\begin{equation}\label{eq:zeta.graphe.p(n)}
\z_{st}(z) = \mathop{\sum}\limits_{n = 1}^{\infty}{bVn^{b-1} +
p(n)\over n^z} = bV\zeta(z-b+1) +
\mathop{\sum}\limits_{l=0}^{b-1}\bigg{(}\mathop{\sum}\limits_{n =
1}^{\infty}{p_l(n)\over n^{z-l}}\bigg{)} , \ \ \text{Re}z > b.
\end{equation}
Define a basis  $\{\alpha_k\}_{k=1}^m$ for the space of
$m$-periodic functions on $\N$, by the formula
$$
\alpha_k(n) = 1 \mbox{  if } n \equiv k\mod m \mbox{ and }0 \mbox{
otherwise, for }k=1,2, \dots , m.
$$
 Set  $p_l(n) = \mathop{\sum}\limits_{k=1}^m
p_{lk}\alpha_k(n)$, with $p_{lk} \in Q$. From
(\ref{eq:zeta.graphe.p(n)})  it follows that
$$
\z_{st}(z) - bV\zeta(z-b+1) =
\mathop{\sum}\limits_{l=0}^{b-1}\bigg{(}\mathop{\sum}\limits_{k=1}^m
p_{lk}\mathop{\sum}\limits_{n = 1}^{\infty}{\alpha_k(n)\over
n^{z-l}}\bigg{)} =
\mathop{\sum}\limits_{l=0}^{b-1}m^{l-z}\bigg{(}\mathop{\sum}\limits_{k=1}^m
p_{lk}\zeta\Big{(}z-l; {k\over m}\Big{)}\bigg{)},
$$
which entails (\ref{eq:zeta.graphe.1}), thus proving part 3) of
the theorem. Part 2) now follows from  (\ref{eq:zeta.graphe.1}),
and the fact that the right-hand side in (\ref{eq:zeta.graphe.1})
doesn't have a pole at $z = b$, because
$\mathop{\sum}\limits_{k=1}^m p_{b-1}(k) = 0$.
\end{proof}
\begin{remark}If $k < m$ are mutually prime integers, it is well known
(see \cite{Apostol76}, Chapter 12)  that Hurwitz' zeta function $\zeta(s, {k \over m})$
may be expressed as a linear combination  of Dirichlet's
L-functions. Thus Formula (\ref{eq:zeta.graphe.1}) may be
rephrased in terms of Dirichlet's L-functions.
\end{remark}


\section{Analytical properties of systolic zeta functions}\label{ch:proprietes.zeta}


For any finite Riemannian polyhedron $(P, g)$ the systolic zeta
function $\z_{\sys}(z)$ defined in Section  \ref{ch:zeta}
satisfies a property analogous to
\ref{prop:zeta_B}. 
Let  $P$ be a finite simplicial polyhedron and set
$$
\Te(P) = \big{\{}(H_1(P, \Z)/\text{Tors}) \setminus \{0\}\big{\}}.
$$
For a Riemannian metric  $g$ on $P$ we consider the length
spectrum $L_{\Te} = \{l_{\te}\}_{\te\in \Te}$  and denote  $V(g) =
\V (B(1))$ the volume of the unit ball of the stable norm in
$H_1(P, \R)$.
\begin{theorem}\label{th:zeta.syst.A}
For any finite Riemannian polyhedron $(P, g)$ the systolic zeta
function $\z_{sys}(z)$ satisfies:

\noindent 1) the series
$$
\z_{sys}(z)=\mathop{\sum}\limits_{\te \in \Te}\frac{1}{l_{\te}^z}
$$
 converges for all $z$ such that  $\text{Re}(z)
> b_1(P)$ and diverges if  $Re(z) < b_1(P)$.

\noindent 2) the function  $\z_{sys}(z)$ extends analytically to
the half-plane $\text{Re}(z) > b_1-1$, with a simple pole at
$z=b_1$ with residue $Res_{b_1}(\z_{sys}(z)) = b_1V(g)$.
\end{theorem}
The proof uses two tools  : Proposition \ref{prop:zeta_B}
and the following lemma originally due to D. Burago
\cite{Burago92}.
\begin{lemme}\label{lemme:Burago}
For any finite Riemannian polyhedron $(P, g)$ there exists  $C>0$
such that
$$
\big{|}l_{\te} - \|\te \|_{st}\big{|} \leq C, \ \ \te \in H_1(P,
\Z)/\text{Tors}.
$$
\end{lemme}
The proof given in  \cite{Burago92} (Theorem  1) applies to the
particular case of the torus $T^m$, one of the lemmata (Lemma 1)
does not extend to the general case. A proof of the general case
is published in   \cite{Cer.Sambusetti2014}.
\begin{proof}[Proof of Theorem \ref{th:zeta.syst.A}]Set $l_{\te} = \|\te\|_{st} +
\a(\te)$, Lemma \ref{lemme:Burago} entails
$$
|\a(\te)| \leq C, \ \ \te \in \Te
$$
from which follows the estimate (\ref{eq:2.convergence}) :
$$
|\Te_{\leq t}| = \O\big{(}t^{b_1(P)}\big{)}.
$$
Therefore the series (\ref{eq:zeta}) converges if  $\mbox{Re}(z) >
b_1(P)$ so $\z_{sys}(z)$ is analytical in the half-plane $\mbox{Re}(z) >
b_1(P)$. On the other hand the series $\z_{st}(z)$ diverges if
$\mbox{Re}(z) < b_1(P)$, whence  $\z_{sys}(z)$ also diverges for $\mbox{Re}(z) <
b_1(P)$, which proves 1). Now observe that
$$
\z_{sys}(z) - \z_{st}(z) =
\mathop{\sum}\limits_{\te\in\Te}\bigg{(}\frac{1}{l_{\te}^z} -
\frac{1}{\|\te\|_{st}^z}\bigg{)},
$$
so, for  $z \in \R$,
\begin{equation}\label{eq:theorem}
\z_{sys}(z) - \z_{st}(z) = -z
\mathop{\sum}\limits_{\te\in\Te}\bigg{(}\frac{\a(\te)}{(\|\te\|_{st}
+ \b(\te))^{z+1}}\bigg{)},
\end{equation}
with $|\b(\te)| < |\a(\te)| \leq C$. The series (\ref{eq:theorem})
converges for  $\mbox{Re}(z) > b_1 -1$, so Proposition \ref{prop:zeta_B}
implies 2), which ends the proof.
\end{proof}
The systolic zeta function  $\z_{sys}(z)$ of a Riemannian
polyhedron  $(P, g)$ encodes a lot of information about $(P, g)$.
For instance the ordered homology length spectrum  $|L_{\Te}|$ may
be recovered from $\z_{sys}(z)$. It follows from a  classical
result of  \cite{Hardy1915} (Theorem 13)  that for all $t \notin
L_{\Te}$ and $c
> b_1(P)$ we have
\begin{equation}\label{eq:integrale}
|\Te_{\leq t}| = {1\over 2\pi
i}\mathop{\int}\limits_{c-i\infty}^{c+i\infty}\z_{sys}(z)e^{tz}{dz\over
z}.
\end{equation}
The integral  (\ref{eq:integrale}) is non-decreasing and piecewise
constant as a function of $t$. The discontinuities of this function
are exactly the points of the ordered length spectrum, and  the
jumps are the corresponding multiplicities.

On the other hand the metric itself cannot be recovered from
$\z_{sys}(z)$. Milnor (\cite{Milnor1964}) gives an example of two
isospectral  (therefore having the same systolic zeta function)
non-isometric 16-dimensional  flat tori, using two  lattices in
$\R^{16}$ discovered by Witt \cite{Witt1941}.

In view of  formula (\ref{eq:integrale}),  an important metric
invariant of $(P, g)$ is
$$
r(P, g) =\inf \{ t :  \Z\langle \Te_{\leq t} \rangle = H_1(P,
\Z)/\text{Tors}\}.
$$
In plain language $r(P, g)$ is the minimal possible length of a
basis of the $\Z$-module  $H_1(P, \Z)/\text{Tors}$. It is known
(see \cite{GLP}, Prop. 5.28) that $r(P, g) \leq 2\text{diam}(P,
g)$.


\begin{ex}\label{ex:zeta.min}
Here we consider a closed, orientable Riemannian manifold $(M,g)$
of  dimension two. It is known (\cite{Massart_these},
\cite{Massart97}) that for all $\te \in \Gamma = H_1(M, \Z)$ there
exist closed geodesics $\gamma_1,\ldots \gamma_k$, $k \leq
b_1(M)/2$,  and integers $\lambda_1,\ldots \lambda_k$,  such that
$\te= \sum_{i=1}^k \lambda_i \left[ \gamma_i \right]$ et
\begin{equation}\label{eq:somme_gamma_i}
\left\|\te \right\|_{st}= \sum_{i=1}^k  \lambda_i \|\left[
\gamma_i \right]\|_{st} = \sum_{i=1}^k  \lambda_i l_g(\gamma_i).
\end{equation}
This motivates our interest in the family
$$
\Te_{min} := \{ \left[\gamma\right] : \mbox{ closed geodesic such that } \|
\left[\gamma\right] \|_{st}=  l_g(\gamma)  \}
$$
and the corresponding zeta function
$$
\zeta_{min}(z) := \sum_{\gamma \in \Te_{min} } l_g(\gamma)^{-z}.
$$
From the obvious inclusion
$$
\{ l_g(\gamma) \ | \gamma \in \Te_{min} \} \subset \{ \|\te\|_{st}
\ | \te \in  \Gamma \}
$$
follows the  convergence of the series $\zeta_{min}(z)$ in the
half-plane  $\text{\mbox{Re}}(z) > b_1(M)$.

On the other hand, from  \cite{Massart_Parlier} we know that
$\zeta_{min}(z)$ cannot converge on any half-plane properly
containing $\text{Re}(z) > 2$. Examples of long-necked surfaces
are studied in \cite{Massart_Parlier}, for such surfaces
$\zeta_{min}(z)$ converges on  $\text{Re}(z) > 2$, and extends
analytically to $\text{Re}(z) > 1$.

It would be interesting to determine the domain of convergence of
$\zeta_{min}(z)$. Another question of interest is whether, in the
case of a hyperbolic metric $g$, $\zeta_{min}(z)$  depends
analytically on the metric, thus defining an analytical function
on the product of Teichm\"uller space with a half-plane.
\end{ex}
\begin{ex}\label{ex:fonction.theta}
Let  $(P, g)$ be a Riemannian polyhedron and let  $\z_{sys}(z)$ be
its systolic zeta function. Applying Mellin's transform  (see
\cite{Hardy1915}, Theorem 11) we get  :
$$
\z_{sys}(z)=\mathop{\sum}\limits_{\te \in \Te}\frac{1}{l_{\te}^z}
= {1\over \Gamma (z)}\mathop{\int}\limits_0^{\infty}\bigg{(}
x^{{z\over 2}-1}\mathop{\sum}\limits_{\te \in \Te}e^{-l^2_{\te}x}
\bigg{)}dx,
$$
for $\text{Re}(z) > b_1(P)$, where $\Gamma (z)=
\int^{+\infty}_{0}t^{z-1}e^{-t}dt$. If  $P = T^b =R^n/\Lambda$ and
$g$  is a flat metric, the function
$$
{\bf \Theta}_g\Big{(}{ix \over \pi}\Big{)} = 1 +
\mathop{\sum}\limits_{\te \in \Te}e^{-l^2_{\te}x}
$$
it the theta function of the lattice $\Lambda$ (see
\cite{Serre70}). For a  $b$-dimensional flat torus $\R^n/\Lambda$,
we obtain the following equality :
$$
\z_{sys}(z) = \z_{st}(z) = {1\over \Gamma
(z)}\mathop{\int}\limits_0^{\infty} x^{{z\over 2}-1}\bigg{(}{\bf
\Theta}_g\Big{(}{ix \over \pi}\Big{)} - 1 \bigg{)}dx,
$$
where  $\text{Re}(z) > b$. 
The most interesting case is
when the lattice is even and unimodular, then ${\bf \Theta}_g$ is
a modular form of weight $b \over 2$, see \cite{Serre70}.
\end{ex}

\section{The Zeta map}
Let $(M, g)$ be a Riemannian manifold and let  $\z_{sys}(z)$ be its
systolic zeta function. In this section we investigate some
properties of the  mapping
\begin{equation}\label{eq:Z-application}
Z: g \longrightarrow \z_{sys}(z)
\end{equation}
from the set of continuous Riemannian metrics on  $M$, which we
denote by $\M(M)$, to the space of holomorphic functions in the
half-plane $\{\mbox{Re}(z) > b_1(M)\}$.

First, let us endow the set  $\M(M)$ with a distance function. Let
$g_i \in \M(M), i=1, 2$ be continuous Riemann metrics on $M$, and
set
\begin{equation}\label{eq:metric.sur.metrics}
\varrho(g_1, g_2) = \mathop{\sup}\limits_{q\in M}\big{(}\inf\{\rho
\in \R_+ \hskip4pt \big{|} \hskip4pt e^{-\rho}\|v_q\|_{g_1}\leq
\|v_q\|_{g_2}\leq e^{\rho}\|v_q\|_{g_1}; v_q \in T_qM \} \big{)}.
\end{equation}
It is easily seen that  $\varrho$ makes $\M(M)$ a complete metric
space. If $\gamma(t), t\in [0, 1]$  is a piecewise smooth curve in
$M$,  and $\rho=\varrho(g_1, g_2)$,  then the $g_i$-lengths of
$\gamma$ are related by the inequalities
$$
e^{-\rho}l^{g_1}(\gamma(t)) \leq l^{g_2}(\gamma(t)) \leq
e^{\rho}l^{g_1}(\gamma(t)).
$$
It follows that the  $g_i$-homology length spectra are
$\varrho$-close:
$$
e^{-\rho}l^{g_1}_{\te} \leq l^{g_2}_{\te} \leq
e^{\rho}l^{g_1}_{\te}, \ \ \te \in \Te(M).
$$

Denote by $\H(b)$ the space of homolorphic functions in the
half-plane $\{ \mbox{Re}(z) > b \}$, topologized by uniform convergence on
compact sets. The above inequality implies the following:
\begin{prop}\label{prop:application.Z.1}
The mapping  $Z: \M(M) \longrightarrow \H(b_1(M))$ defined in
(\ref{eq:Z-application})  is continuous.
\end{prop}

Let  $(\R^b, \Gamma)$ be a vector space with a lattice. To any
Banach space structure  $\B = (\R^b, \|\cdot \|)$  on $\R^b$ we
associate, as  in Section  \ref{ch:zeta.stable}, the function
$\z_{\B}(z)$. Fixing $\Gamma$, and varying the norm $ \|\cdot \|$,
we define as in  (\ref{eq:Z-application}) a map
\begin{equation}\label{eq:Z-Banach-application}
Z_B: \B \longrightarrow \z_{\B}(z).
\end{equation}

Denote by $B = B(b)$ the set of all norms on $\R^b$. Any two  norms
$\B_i = (\R^b, \|\cdot \|_i), \ \ i=1,2$ are equivalent. Let
$c_{12}$ be the minimal  constant such that
$$
c_{12}^{-1}\|{\bf x}\|_1 \leq \|{\bf x}\|_2 \leq c_{12}\|{\bf
x}\|_1 \ \ {\bf x}\in \R^b.
$$
Then  $\rho(\B_1, \B_2):= \ln c_{12}$  defines a distance on $B$.
The map (\ref{eq:Z-Banach-application}) is easily seen to be
continuous.

Now we would like  to describe the image $I(b) = \text{im}(Z_B)$
of the map (\ref{eq:Z-Banach-application}). The group $GL(b)$
operates on  $B$ by the rule  $\| {\bf x} \|_h = \| h^{-1}({\bf
x}) \|, \ \ h \in GL(b)$. 
 This action transfers to zeta functions is such a way that
 $Z_B$ is  equivariant.
The set   $Q(b) = B(b)/ GL(b)$ of isomorphism classes of  Banach
structures on  $\R^b$, called the  Banach-Masur space, is  compact
with respect to its natural metric.  The quotient map
$$
\hat{Z_B}: Q(b) \longrightarrow I(b)/ GL(b)
$$
is onto. We have just proved the following
\begin{theorem}\label{th:zeta.syst}
The group $GL(b)$ operates in a natural way on  $I(b)$. The
quotient set  $I(b) /GL(b)$  is  compact. For any $f \in I(b)$ and
any  $h \in GL(b)$ we have
$$
Res_b f_h = ({\det}h) Res_b f.
$$

\end{theorem}
Taking into account the fact  that the  right hand side of
(\ref{eq:theorem}) belongs to $\H(b_1(M)-1)$, we obtain the
\begin{cor}
The image of $Z: \M(M) \longrightarrow \H(b_1(M))$ is locally
compact modulo $\H(b_1(M)-1)$.
\end{cor}
Note that the compactness is only local because  $GL(b)$ is not
compact.


\section{Zeta functions and isoperimetric inequalities with the length spectrum}\label{ch:isoperimetrique}
In this section we consider a closed, orientable manifold  $M$ of
dimension $m$  and first  Betti  number $b =b_1(M)$. First we
define a few algebraic invariants of $M$, see \cite{Babenko92} for
more details. The characteristic map
(\ref{eq:application.caracteristique.bis}) induces a map between
integer homology groups
$$
f_{\ast}: H_m(M, \Z) \longrightarrow H_m(T^b, \Z) \simeq
{\Lambda}_mH_1(T^b, \Z).
$$
 Set $\a(M) = f_{\ast}([M])$ where  $[M]$ is the fundamental class of  $M$. The  Haar measure on $H_1(T^b, \R)$, henceforth denoted by 
$\Vol$,  is normalised in such a way that the lattice $H_1(T^b,
\Z)\subset H_1(T^b, \R)$ has volume 1. Take a basis  ${\bf e} =
\{e_1, \dots, e_b\}$ de $H_1(T^b, \R)$, and decompose
\begin{equation}\label{eq:alpha(M)}
\a(M) = \mathop{\sum}\limits_{i_1 < \dots < i_m}\a_{i_1  \dots
i_m}e_{i_1}\wedge \dots \wedge e_{i_m}.
\end{equation}

Let us say that  ${\bf e}$ is {\it subordinate} to  $M$ if the
inequality $|\a_{i_1 \dots i_m}| \leq 1$ holds for every
coefficient of  (\ref{eq:alpha(M)}). Denote by $E(M)$ the set of
basis subordinate to  $\a(M)$. For any basis ${\bf e}$, let
$\Vol({\bf e})$ be the  volume of the  solid generated by the
vectors in ${\bf e}$. We define the {\it algebraic volume} of  $M$
by setting
$$
V[M] = V(\a(M)) = \mathop{\inf}\limits_{{\bf e}\in E(M)} \Vol({\bf
e}).
$$
Similarly we define the {\it algebraic mass} of  $M$ by
$$
m[M] = m(\a(M)) = \mathop{\inf}\limits_{\Vol({\bf
e})=1}\bigg{(}\mathop{\sum}\limits_{i_1 < \dots < i_m}|\a_{i_1
\dots i_m}| \bigg{)}.
$$
\begin{remark}\label{rk:polyvecteur}
1. Here is the reason for using the word "mass". Endow $H_1(T^b,
\R)$ with a scalar product  $\langle \hskip3pt , \rangle$ and
assume the Haar measure it induces coincides with  $\Vol$. Let
$\mass(\a)$ be the usual mass norm defined on ${\Lambda}_mH_1(T^b,
\R)$  by means of $\langle \hskip3pt , \rangle$
(see\cite{Federer69}). It is easily proved that for all $\a \in
{\Lambda}^mH_1(T^b, \R)$ we have
$$
m(\a) = \mathop{\inf}\limits_{\langle \hskip3pt , \hskip3pt
\rangle}\mass(\a),
$$
where  $\langle \hskip3pt , \rangle$ ranges over the set of scalar
products of volume 1.

\noindent 2.The algebraic volume  $V[M]$ may vanish for manifolds
which are homologically essential, that is,
 $f_{\ast}([M]) \neq 0$.  The invariant $V[M]$ measures the  maximal non-degeneracy
 of $f_{\ast}([M])$  as a polyvector of  $H_1(T^b, \R)$.

\noindent 3. We have the following universal inequality between
algebraic mass and algebraic volume
 (see \cite{Babenko92} for details):
$$
\forall \a \in {\Lambda}^mH_1(T^b, \R),\ \big{(}m(\a)\big{)}^b
\leq \left(
\begin{matrix} b \\ m
\end{matrix}
\right)^b \big{(}V(\a)\big{)}^m.
$$

\noindent 4. When  $m = b_1$ everything boils down to the degree
of the characteristic map $f$:
$$
m[M] = V[M] = |\deg(f)|.
$$
\end{remark}
\begin{ex}\label{ex:vol.algebrique}
Let $M = M_h$ be an  orientable surface of genus $h$. The  class
$\a(M_h) \in H_2(T^{2h}, \Z)$  is induced by a symplectic form
with integer coefficients on $H_1(T^{2h})$. 
It is proved in  \cite{Babenko92} that
$$
h \geq m[M_h] \geq (h!)^{1\over h}; \ \ V[M_h] \geq
{1\over(2h-1)!!},
$$
where $(2h-1)!! = 1\cdot 3\cdot 5 \cdot ... \cdot (2h-1) $.
\end{ex}
\subsection{The universal constant $v_b$}
Let us consider the following variational problem  from convex
geometry (for more about this  see \cite{Babenko88} and - with a
diffferent   renormalisation - \cite{Babenko92}). Let $B$ be a
centrally symmetric convex body  (CSCC) in  $(\R^b, \Vol)$, denote
$\mathcal{P}(B)$ the set of parallelepipeds which contain  $B$.
Define
$$
v_{supp}(B) = {1 \over {2^b}}\mathop{\inf}\limits_{P\in
\mathcal{P}(B)}\Vol(P).
$$
The universal constant we are interested in is defined as
$$
v_b = \mathop{\inf}\limits_{B}\frac{\Vol(B)}{v_{supp}(B)},
$$
where  $B$ ranges over all CCSC's with non-empty interior.

The  constant $v_b$ was introduced in asymptotic geometry in
\cite{Babenko88}. At the same time it  appeared in different
contexts, among which geometric number theory, see the
introduction of \cite{Pelczynski91} for more details.

The asymptotic rate of $v_b$ in $b$ is rather  simple, see
\cite{Kashin89}. There are two positive constants $c, C$ such that
$$
c^b \leq {v_b \over \omega_b} \leq C^b,
$$
where $\omega_b$ is the volume of the unit ball in $\R^b$. The
optimal values of $c$ and $C$ are not known. For practical
purposes one can use $c = {1\over {\sqrt e}}$ and $C = 1$, see
\cite{Pelczynski91}.

The exact value of  $v_b$ is known only when $b=2$ or $v_2 = 3$,
with the infimum achieved at affine regular hexagons
\cite{Babenko88}. An asymptotically correct lower estimate in $b$
is obtained in \cite{Pelczynski91}.

\subsection{Isoperimetric inequality with the length spectrum}
Now let us consider an orientable closed manifold  $M$ of
dimension $m$ with first   Betti number  $b_1 = b_1(M)$. Let $g$
be a Riemannian metric on  $M$, and let  $\z_{sys}(z)$ be its
systolic zeta function.
\begin{prop}\label{th:isoperimetrique}
For any Riemannian metric  $g$ on  $M$  we have the following
inequality:
$$
\bigg{(}\frac{b_1v_{b_1}V[M]}{Res_{b_1}(\z_{sys}(z))}\bigg{)}^m
\leq \big{(}\vol(M, g)\big{)}^{b_1}.
$$
\end{prop}
Manifolds for which  $\dim(M) = b_1(M)$ are a very interesting
class. The  $m$-torus is the simplest example but they may have
much more complex topology, for instance the degree of the
characteristic map (\ref{eq:application.caracteristique.bis}) may
take any integer value. Proposition \ref{th:isoperimetrique} and
Remark \ref{rk:polyvecteur} combine to yield the
\begin{cor}
Let  $M$ be an orientable manifold of dimension and first Betti
number  $m$. For any Riemannian metric  $g$ on $M$ we have
$$
mv_m|\deg(f)| \leq Res_{b_1}(\z_{sys}(z))\vol(M, g),
$$
where  $f$ is the characteristic map.
\end{cor}
It follows from Proposition   \ref{th:isoperimetrique},  Example
\ref{ex:vol.algebrique}, and the lower estimate  for  $v_{2h}$
given in  \cite{Pelczynski91} that
\begin{cor}
 For any Riemmanian metric $g$ on an  orientable surface $M_h$
of genus  $h$ we have
$$
\frac{2h \pi^h}{h!(2h-1)!!}\bigg{(}h + {1\over 2}\bigg{)}^h \left(
\begin{matrix} h(2h+1) \\ 2h
\end{matrix}
\right)^{-{1\over 2}} \leq Res_{2h}(\z_{sys}(z))\big{(}\vol(M_h,
g)\big{)}^h .
$$
\end{cor}
\noindent For large genera we deduce the following asymptotic
estimate :
$$
\frac{\pi e}{2h} \lesssim
\big{(}Res_{2h}(\z_{sys}(z))\big{)}^{1\over h}\vol(M_h, g), \ \
\text{si} \ \ h \gg 1.
$$

\begin{proof}[Proof of Proposition  \ref{th:isoperimetrique}]
Consider the homology (or Abelian) cover  $\widetilde{M}$ of $M$,
that is, the cover whose transformation group is   $\Gamma =
H_1(M, \Z)/\mbox{Tors}$, so   $\widetilde{M}/\Gamma = M$. Let
$\widetilde{g}$ be the lift to $\widetilde{M}$ of the metric $g$.
Fix a  point $q \in \widetilde{M}$ and denote by
$$
V_g(q, t) = \{x \in \widetilde{M}\big{|} \dist_{\widetilde{g}}(q,
x) \leq t \}
$$
 the metric ball in $\widetilde{M}$ with center  $q$ and radius $t$. It is well-known that the following limit
$$
\Omega_H(M, g) = \mathop{\lim}\limits_{t \rightarrow
\infty}\frac{\vol(V_g(q, t), \widetilde{g})}{t^{b_1}}
$$
exists  and does not depend on the choice of $q$. It is easily
seen that
\begin{equation}\label{eq:volume.asymptotique}
\Omega_H(M, g) = V(g)\vol(M, g),
\end{equation}
where  $V(g)$  is the volume of the unit ball of the stable norm
on $H_1(M, \R)$. The quantity $\Omega_H(M, g)$ is called
asymptotic homology volume.  In \cite{Babenko92} a lower estimate
for  $\Omega_H(M, g)$ was sought.

Set $s = {b_1-m\over m}$. It is easily seen that
$$
\Omega_H(M, g)\big{(}\vol(M, g)\big{)}^s= V(g)\big{(}\vol(M,
g)\big{)}^{b_1\over m}
$$
is invariant under dilatation of the metric: by a factor  $\lambda^2$. It is proved in  \cite{Babenko92},
 Theorem 5.2, that for any metric  $g$  on $M$ we have
$$
v_{b_1}V[M] \leq \Omega_H(M, g)\big{(}\vol(M, g)\big{)}^s.
$$
This inequality and (\ref{eq:volume.asymptotique}) combine to end
the proof.
\end{proof}

\bigskip

{

\end{document}